\documentclass{article}
\usepackage{amsmath}
\usepackage{amssymb}

\newtheorem{theorem}{Theorem}
\newtheorem{acknowledgement}[theorem]{Acknowledgement}

\newtheorem{conjecture}[theorem]{Conjecture}

\newtheorem{remark}[theorem]{Remark}

\newenvironment{proof}[1][Proof]{\noindent\textbf{#1.} }{\ \rule{0.5em}{0.5em}}
\input{tcilatex}

\begin{document}

\title{The strong renewal theorem with infinite mean via local large
deviations.}
\author{R. A. Doney \and University of Manchester}
\date{}
\maketitle

\begin{abstract}
A necessary and sufficient condition is established for an asymptotically
stable renewal process to satisfy the strong renewal theorem. This result is
valid for all $\alpha \in (0,1),$ thus completing a result for $\alpha \in
(1/2,1)$ which was proved in the 1963 paper of Garsia and Lamperti \cite{GL}%
. This paper is superseded by arXiv:1612.07635.
\end{abstract}

\section{Introduction}

This paper contains new results about asymptotically stable random walk in
two different, but related areas. The first result, which applies to any
random walk $S=\{S_{n},n\geq 0\}$ which is in the domain of attraction of a
stable law of index $\alpha \in (0,1)\cup (1,2)$ without centering, is a
local large deviation bound which improves the error term in Gnedenko's
local limit theorem without making any further assumptions. This bound may
have other uses, but here we use it to give a complete answer to a question
which has remained open since the 1963 paper of Garsia and Lamperti \cite{GL}%
; viz which renewal processes in the domain of attraction of a stable law of
index $\alpha \in (0,1)$ verify the strong renewal theorem (SRT). We also
give an answer to the same question for certain generalized renewal
functions, and indicate how the renewal process proof can be modified to get
the corresponding result for asymptotically stable subordinators. In the
main part of the paper, dealing with renewal processs, we will restrict
attention to the case of an aperiodic distribution on the integer lattice,
but it is easy to see that the non-lattice case can be treated by the same
techniques. This paper is superseded by arXiv:1612.07635.

\section{Results}

We write $S_{0}=0,S_{n}=\sum_{1}^{n}X_{r},$ the $X^{\prime }s$ being i.i.d.
with mass function $p$ and distribution function $F.$ We also put $P(X>x)=%
\overline{F}(x),$ and $X_{n}^{\ast }=\max_{1\leq r\leq n}|X_{r}|.$

\begin{theorem}
\label{LOCAL}Suppose $S_{n}/a_{n}$ converges in distribution to a stable $%
(\alpha ,\rho )$ law, where $\alpha \in (0,1)\cup (1,2)$ and the positivity
parameter $\rho $ is positive. Then, given any $\gamma >0,$ $\exists
C_{0},n_{0},\theta _{0}$ such that, for all $n\geq n_{0}$ and $x\geq n\theta
_{0},$%
\begin{equation}
P\{S_{n}=x,X_{n}^{\ast }\leq \gamma x\}\leq \frac{C_{0}\{n\overline{F}%
(x)\}^{1/\gamma }}{a_{n}}.  \label{l1}
\end{equation}
\end{theorem}

and

\begin{equation}
P\{S_{n}=x\}\leq C_{0}\frac{n\overline{F}(x)}{a_{n}}.  \label{l2}
\end{equation}

\begin{remark}
A result similar to this for the case $\alpha =2,EX^{2}<\infty $ can be
found in Lemma 4 of \cite{nag}.
\end{remark}

\begin{remark}
Gnedenko's local limit theorem implies that $a_{n}P\{S_{n}=x\}\rightarrow 0$
whenever $\theta :=x/a_{n}\rightarrow \infty ,$ but gives no information
about the rate. This is given by (\ref{l2}), since we have%
\begin{equation*}
n\overline{F}(x)\backsim c\frac{\overline{F}(x)}{\overline{F}(a_{n})}=c\frac{%
\overline{F}(a_{n}\theta )}{\overline{F}(a_{n})},
\end{equation*}%
so that if we write $\Lambda =(n\overline{F}(x))^{-1}$ it follows from
Potter's bounds that for any $\varepsilon >0$ we have%
\begin{equation}
c\theta ^{\alpha -\varepsilon }\leq \Lambda \leq c\theta ^{\alpha
+\varepsilon }  \label{ly}
\end{equation}%
for all sufficiently large $n$ and $\theta .$ (Here, and in what follows, $c$
denotes a generic positive constant whose value can change from line to
line.)
\end{remark}

Our main application of this is to prove the second part of the following,
in which we write $g(x)=\sum_{0}^{\infty }P(S_{n}=x)$ for the renewal mass
function.

\begin{theorem}
\label{srt}(i) Assume that $F$ is aperiodic, $S\in D(\alpha ,\rho )$ with $%
\alpha \in (1/2,1)$ and $\rho >0.$ Then 
\begin{equation}
\lim_{x\rightarrow \infty }x\overline{F}(x)g(x)=g(\alpha ,\rho ):=\alpha
E(Y^{-\alpha }:Y>0),  \label{r2}
\end{equation}%
where $Y$ denotes a random variable having the limiting stable law.

(ii) Assume that $F$ is aperiodic, $P(X\geq 0)=1,$ and $S\in D(\alpha ,1)$
with $\alpha \in (0,1/2]$. Then (\ref{r2}) holds with $\rho =1$ if and only
if%
\begin{equation}
\lim_{x\rightarrow \infty }x\overline{F}(x)p(x)=0,  \label{rz}
\end{equation}%
and 
\begin{equation}
\lim_{\delta \rightarrow 0}\lim \sup_{x\rightarrow \infty }x\overline{F}%
(x)\sum_{1}^{\delta x}\frac{p(x-w)}{w\overline{F}(w)^{2}}=0.  \label{r3}
\end{equation}
\end{theorem}

\begin{remark}
The condition (\ref{rz}) is easily seen to be equivalent to%
\begin{equation}
\lim_{x\rightarrow \infty }x\overline{F}(x)\sum_{1}^{n_{0}}P(S_{n}=x)=0\text{
for any fixed }n_{0},  \label{r1}
\end{equation}%
and we will use this repeatedly.
\end{remark}

\begin{remark}
The statement (\ref{r2}) is called the SRT and is the obvious analogue of
the Renewal Theorem when the mean is infinite. The question as to which
asymptotically stable renewal processs satisfy it has been extensively
studied since the \ pioneering paper of Garsia and Lamperti \cite{GL}, who
first established (i) above. They only considered the "lattice renewal
case", i.e. they assumed $P(X\in \boldsymbol{Z}^{+})=1,$ but their results
were extended to the case of a general random walk in \cite{E} and \cite{W}.
In the case $\alpha \leq 1/2$ it is easy to show that (\ref{r1}) is
necessary for the SRT to hold, and the papers \cite{W}, \cite{D}, \cite{V}, 
\cite{C1} and \cite{C2} contain a succession of sufficient conditions for
the SRT to hold, based on restictions on the asymptotic behaviour of the
ratio $xp(x)/\overline{F}(x),$ and its non-lattice counterpart.
\end{remark}

\begin{remark}
When $\alpha \in (1/2,1)$ the fact that for fixed \thinspace $n$ we have $%
P(S_{n}=x)\leq P(S_{n}>x-1)\backsim n\overline{F}(x)$ shows that (\ref{rz})
holds and the fact that $1/(x\overline{F}(x)^{2})$ is asymptotically
increasing shows that%
\begin{equation*}
\lim_{x\rightarrow \infty }x\overline{F}(x)\sum_{1}^{\delta x}\frac{p(x-w)}{w%
\overline{F}(w)^{2}}\leq \lim_{x\rightarrow \infty }\frac{cx\overline{F}(x)}{%
\delta x\overline{F}(\delta x)^{2}}\cdot \left( \overline{F}((1-\delta )x)-%
\overline{F}(x)\right) =\delta ^{2(1-\alpha )},
\end{equation*}%
so (\ref{r2}) also holds. So (ii) is technically also correct for $\alpha
\in (1/2,1).$
\end{remark}

\begin{remark}
In the case $\alpha =1/2,$ it is easy to check that if we put $\overline{F}%
(x)=(xL(x))^{-1/2}$ then both (\ref{rz}) and (\ref{r3}) hold if $%
\lim_{x\rightarrow \infty }L(x)>0,$ so only in case $\lim_{x\rightarrow
\infty }L(x)=0$ is the NASC required. Thus the case $\overline{F}(x)\backsim
(cx)^{-1/2}$ represents the boundary of the situation where the SRT holds
without further conditions.
\end{remark}

\section{Proof of Theorem \protect\ref{LOCAL}}

\begin{proof}
If we write $P\{S_{n}=x\}=P_{1}+P_{2},$ where $P_{2}=P\{S_{n}=x,X_{n}^{\ast
}>\gamma x\},$ the simple estimate%
\begin{equation*}
P_{2}\leq n\sum_{|z|>\gamma x}p(z)P(S_{n-1}=x-z)\leq \frac{cn\{\overline{F}%
(\gamma x)+F(-\gamma x)\}}{a_{n}}\leq \frac{cn\overline{F}(x)}{a_{n}},
\end{equation*}%
which is a consequence of Gnedenko's local limit theorem, shows that (\ref%
{l2}) follows from (\ref{l1}), which we now prove.

We introduce an associated distribution $\tilde{P},$ by setting 
\begin{equation}
\tilde{p}(z)=\tilde{P}(X_{1}=z)=e^{\mu z}p(z)\boldsymbol{1}_{(|z|\leq \gamma
x)}/m_{0},  \label{l3}
\end{equation}%
where $m_{0}=\sum_{|z|\leq \gamma x}e^{\mu z}p(z)$ and we set $\mu :=\frac{%
\log \Lambda }{\gamma x}.$ Note that $\Lambda \rightarrow \infty $ as $%
n,\theta \rightarrow \infty ,$ and, since $e^{\gamma \mu x}=\Lambda ,$
iteration of (\ref{l3}) gives 
\begin{equation}
P_{1}=m_{0}^{n}\Lambda ^{-\frac{1}{\gamma }}\tilde{P}(S_{n}=x).  \label{l4}
\end{equation}%
We start by showing that $m_{0}^{n}\leq c$, for all sufficiently large $n$
and $\theta .$ When $\alpha \in (0,1),$ we write%
\begin{equation*}
1-m_{0}=\sum_{|z|<1/\mu }(1-e^{\mu z})p(z)+P(|X|\geq 1/\mu )-\sum_{1/\mu
\leq |z|\leq \gamma x}e^{\mu z}p(z),
\end{equation*}%
so that%
\begin{eqnarray*}
|1-m_{0}| &\leq &c\mu \sum_{|z|<1/\mu }|z|p(z)+P(|X|\geq 1/\mu )+\sum_{1/\mu
\leq |z|\leq \gamma x}|z|e^{\mu |z|-\log |z|}p(z) \\
&\leq &c\overline{F}(1/\mu )+\frac{e^{\gamma \mu x}}{\gamma x}\sum_{|z|\leq
\gamma x}|z|p(z),
\end{eqnarray*}%
where we have used the observation that $\mu z-\log z$ is monotone
increasing on $[1/\mu ,\infty ),$ and standard properties of regularly
varying functions. The second term above is bounded by%
\begin{equation}
\frac{ce^{\gamma \mu x}x\overline{F}(x)}{x}=c\Lambda \overline{F}(x)=\frac{c%
}{n},  \label{l6}
\end{equation}%
so that, again by Potter's bounds, for any $\tilde{\alpha}\in (0,\alpha )$%
\begin{eqnarray}
|1-m_{0}| &\leq &\frac{c}{n}(1+n\overline{F}(1/\mu ))  \notag \\
&\leq &\frac{c}{n}(1+\frac{\overline{F}(1/\mu )}{\overline{F}(a_{n})})\leq 
\frac{c}{n}(1+(a_{n}\mu )^{-\tilde{\alpha}})\leq \frac{c}{n},  \label{l5}
\end{eqnarray}%
the last step relying on the fact, which follows from (\ref{ly}), that 
\begin{equation*}
\frac{1}{a_{n}\mu }=\frac{\gamma \theta }{\log \Lambda }\rightarrow \infty .
\end{equation*}%
When $\alpha \in (1,2)$ we have $EX_{1}=0,$ so that%
\begin{eqnarray*}
\sum_{|z|<1/\mu }(1-e^{\mu z})p(z) &=&\sum_{!z|<1/\mu }(1-e^{\mu z}+\mu
z)p(z)+\mu \sum_{|z|\geq 1/\mu }zp(z) \\
&\leq &c\{\sum_{-|z|<1/\mu }\mu ^{2}z^{2}p(z)+F(-1/\mu )+\overline{F}(1/\mu
)\} \\
&\leq &c\overline{F}(1/\mu ),
\end{eqnarray*}%
and we can also write%
\begin{eqnarray*}
\sum_{1/\mu \leq |z|\leq \gamma x}e^{\mu |z|}p(z) &=&\sum_{1/\mu \leq
|z|\leq 2/\mu }e^{\mu |z|}p(z)+\sum_{2/\mu \leq |z|\leq \gamma
x}|z|^{2}e^{\mu |z|-2\log |z|}p(z) \\
&\leq &c\{P(|X|>1/\mu )+\frac{e^{\mu \gamma x}}{(\gamma x)^{2}}%
E(X^{2}:|X|\leq \gamma x)\} \\
&\leq &c\{P(|X|>1/\mu )+\Lambda \overline{F}(x)\}
\end{eqnarray*}%
so again (\ref{l5}) holds. Thus we have $m_{o}^{n}\leq c$ for all
sufficiently large $n$ and $\theta ,$ which gives%
\begin{equation*}
P_{1}\leq c\Lambda ^{-\frac{1}{\gamma }}\tilde{P}(S_{n}=x),
\end{equation*}%
and we are left to prove that $a_{n}\tilde{P}(S_{n}=x)\leq c.$ We do this by
applying a suitable Normal approximation, for which we need to estimate 
\begin{equation*}
m_{k}:=\tilde{E}X_{1}^{k}=\frac{1}{m_{0}}\sum_{|z|\leq \gamma x}z^{k}e^{\mu
z}p(z):=\frac{\tilde{m}_{k}}{m_{0}},
\end{equation*}%
for $k=1,2,3.$ Since $m_{0}\geq \sum_{0\leq z\leq \gamma x}p(z)\rightarrow
\rho >0,$ and (\ref{l6}) gives $m_{0}\leq c,$ it suffices to estimate $%
\tilde{m}_{k};$ first, when $\alpha <1$ we have%
\begin{eqnarray}
|\tilde{m}_{k}| &\leq &e^{\mu \gamma x}E\{|X|^{k}:0\leq |X|\leq \gamma x\}
\label{l7} \\
&\backsim &c\Lambda x^{k}\overline{F}(x)=cx^{k}/n\text{ for }k=1,2,3.  \notag
\end{eqnarray}%
When $\alpha \in (1,2)$ this is also valid for $k=2,3,$ but for $k=1$ a
little more work is required. Specifically we write%
\begin{eqnarray*}
|\sum_{|z|<1/\mu }e^{\mu z}zp(z)| &\leq &\sum_{|z|<1/\mu }|z(e^{\mu
z}-1)|p(z)+|\sum_{|z|<1/\mu }zp(z)| \\
&\leq &c\sum_{|z|<1/\mu }\mu z^{2}p(z)+|\sum_{|z|\geq 1/\mu }zp(z)|\leq c%
\overline{F}(1/\mu )/\mu \leq \frac{x}{n\log \Lambda },
\end{eqnarray*}%
and 
\begin{eqnarray*}
|\sum_{1/\mu \leq |z|<\gamma x}e^{\mu z}zp(z)| &\leq &\sum_{1/\mu \leq
|z|<\gamma x}e^{\mu |z|-\log |z|}z^{2}p(z) \\
\leq \frac{e^{\mu \gamma x}}{\gamma x}E(X^{2}:|X| &\leq &\gamma x)\leq
c\Lambda (x)x\overline{F}(x)=cx/n,
\end{eqnarray*}%
so (\ref{l7}) also holds in this case. We also need a lower bound for $%
\tilde{\sigma}^{2}=\tilde{E}(X-\tilde{m}_{1})^{2}$. Since, for all $x,n$
large enough and for any $d\in (0,1),$ $(y-\tilde{m}_{1})^{2}\geq cy^{2}$
for $y\in (x(1-d)\leq y\leq x)$ we \ get, 
\begin{eqnarray*}
\tilde{\sigma}^{2} &\geq &e^{\mu \gamma x(1-d)}E\{(X-\tilde{m}%
_{1})^{2}:\gamma x(1-d)\leq X\leq \gamma x\} \\
&\geq &c\Lambda ^{1-d}x^{2}\overline{F}(x)=cx^{2}/n\Lambda ^{d}
\end{eqnarray*}%
so that%
\begin{equation}
cx^{2}\geq n\tilde{\sigma}^{2}\geq \frac{cx^{2}}{\Lambda ^{d}}.  \label{l8}
\end{equation}%
If $\tilde{\nu}:=\tilde{E}|X-\tilde{m}_{1}|^{3},$ a similar calculation
leads to the bounds%
\begin{equation}
cx^{3}\geq n\tilde{\nu}\geq \frac{cx^{3}}{\Lambda ^{d}}.  \label{l9}
\end{equation}%
With these results in hand, we can apply Lemma 3 of \cite{D} to deduce that $%
P_{1}$ is bounded above by%
\begin{equation*}
c\left( \frac{1}{\sqrt{n\tilde{\sigma}^{2}}}+\frac{n\tilde{\nu}}{(n\tilde{%
\sigma}^{2})^{2}}+\int_{\tilde{\sigma}^{2}/4\tilde{\nu}}^{\pi }e^{-n(1-|\phi
(t)|)}dt\right) ,
\end{equation*}%
where $\phi (t)=\tilde{E}(e^{itX}).$ By (\ref{l8}) and (\ref{l9}), the first
two terms are bounded above by $cx^{-1}\Lambda ^{2d},$ and this in turn is
bounded asymptotically, for suitably chosen $d,$ by $c/a_{n}.$ We also have
the bound 
\begin{eqnarray*}
1-|\phi (t)| &\geq &1-\mathcal{R}\phi (t)\geq cE(e^{\mu X}(1-\cos
tX):|X|\leq \gamma x) \\
&\geq &ct^{2}E(X^{2}:0\leq X\leq \gamma x\wedge t^{-1})) \\
&\geq &c\overline{F}(t^{-1})\text{ for all }t\geq 1/\gamma x.
\end{eqnarray*}%
For all $n$ such that $a_{n}\leq \gamma x$ we can therefore bound the
integral term by%
\begin{eqnarray*}
&&\int_{\tilde{\sigma}^{2}/4\tilde{\nu}}^{1/a_{n}}dt+\int_{1/a_{n}}^{\pi
}e^{-cn\overline{F}(t^{-1})}dt \\
&\leq &\frac{1}{a_{n}}(1+\int_{1}^{a_{n}\pi }e^{-cn\overline{F}%
(a_{n}z^{-1})}dz).
\end{eqnarray*}%
But by Potter's bounds we have 
\begin{equation*}
n\overline{F}(a_{n}z^{-1})=\frac{\overline{F}(az^{-1})}{\overline{F}(a_{n})}%
\geq c(za_{n})^{\alpha _{1}}
\end{equation*}%
for any $\alpha _{1}\in (0,\alpha ),$ and we deduce that $a_{n}\tilde{P}%
(S_{n}=x)$ is bounded above, and (\ref{l1}) follows.
\end{proof}

\section{Proof of Theorem \protect\ref{srt}}

\begin{proof}
We will introduce a quantity $A(x)=x^{\alpha }L_{0}(x)$ where $L_{0}(x)$ is
a normalised slowly varying function, (see \cite{BGT}, p15) which satisfies $%
A(x)\backsim 1/\overline{F}(x)$ Then $A$ is differentiable and $xA^{\prime
}(x)/A(x)\rightarrow \alpha .$ We can and will take the norming sequence $%
a_{n}$ to be the value at $x=n$ of $a(x),$where $A(a(x))=1,x\geq 1.$

Our first step is to establish that (\ref{r2}) holds iff for some fixed $%
n_{0}$ 
\begin{equation}
\lim_{\delta \rightarrow 0}\lim \sup_{x\rightarrow \infty }\frac{x}{A(x)}%
\sum_{n\in (n_{0},\delta A(x)]}P(S_{n}=x)=0.  \label{r5}
\end{equation}%
To see this, fix any $\delta >0,$ note that $n>\delta A(x)\iff x<a(n/\delta
)\backsim $ $\delta ^{-\eta }a(n),$ so given $\varepsilon >0$ Gnedenko's
local limit theorem allows us to choose $x_{0}(\varepsilon )$ large enough
to ensure that, for all $x\geq x_{0}(\varepsilon ),$ on this range we have 
\begin{equation*}
a_{n}|P(S_{n}=x)-f(x/a_{n})|\leq \varepsilon
\end{equation*}%
where $f$ is the density of $Y$. Then%
\begin{equation*}
\left\vert \sum_{n>\delta A(x)}P(S_{n}=x)-\sum_{n>\delta A(x)}\frac{%
f(x/a_{n})}{a_{n}}\right\vert \leq \varepsilon \sum_{n>\delta A(x)}\frac{1}{%
a_{n}}\leq \varepsilon c\frac{\delta A(x)}{\delta _{x}^{\ast }x},
\end{equation*}%
where $a(\delta A(x))=\delta _{x}^{\ast }x,$ so that $\delta _{x}^{\ast
}\backsim \delta ^{\eta }$ as $x\rightarrow \infty .$ Putting $x/a_{n}=y$ so
that $n=A(x/y),$ we see that the second sum on the left is a Riemann
approximation to%
\begin{eqnarray*}
&&\int_{0}^{x/a(\delta A(x))}\frac{f(y)A^{\prime }(x/y)ydy}{x^{2}}dy\backsim 
\frac{\alpha }{x}\int_{0}^{1/\delta _{x}^{\ast }}f(y)A(x/y)dy \\
&=&\frac{\alpha A(x)}{x}\int_{0}^{1/\delta _{x}^{\ast }}\frac{A(x/y)}{A(x)}%
f(y)dy\backsim \frac{\alpha A(x)}{x}\int_{0}^{\delta ^{-\eta }}y^{-\alpha
}f(y)dy,
\end{eqnarray*}%
where in the last step we used Potter's bounds and dominated convergence.
Since we can choose $\delta $ as small as we like and $\varepsilon $ is
arbitrary, we see that%
\begin{equation}
\lim_{\delta \rightarrow 0}\lim_{x\rightarrow \infty }\frac{x}{A(x)}%
\sum_{n>\delta A(x)}P(S_{n}=x)=\alpha E(Y^{-\alpha }:Y>0)=g(\alpha ,\rho ),
\label{r11}
\end{equation}%
and we are left to ascertain when (\ref{r5}) and (\ref{r1}) are valid.

When $\alpha \in (1/2,1);$we have already noted that (\ref{r1}) holds. Also
in this case we have $n^{2}/a_{n}$ $\in RV(2-\eta ),$ i.e. is regularly
varying with index $2-\eta >0,$ so from (\ref{l2})%
\begin{eqnarray*}
\frac{x}{A(x)}\sum_{n_{0}}^{\delta A(x)}P(S_{n} &=&x)\leq \frac{C_{0}x}{A(x)}%
\sum_{n_{0}}^{\delta A(x)}\frac{n\overline{F}(\text{ }x)}{a_{n}} \\
&\leq &\frac{cx}{A(x)^{2}}\cdot \frac{\delta ^{2}A(x)^{2}}{a(\delta A(x))}=%
\frac{c\delta ^{2}}{\delta _{x}^{\ast }}.
\end{eqnarray*}%
Since $\delta _{x}^{\ast }\backsim \delta ^{\eta }$ as $x\rightarrow \infty $
with $1<\eta <2,$ (\ref{r5}) follows, and we have proved (i).

So from now on we take $\alpha \in (0,1/2]$ and consider only the renewal
case, i.e. we assume $P(X\geq 0)=1.$ Since (\ref{r1}) is obviously
necessary, from now on we will assume that it holds, and show that (\ref{r5}%
) holds iff (\ref{r3}) holds.

Clearly (\ref{r3}) is equivalent to 
\begin{eqnarray}
\lim_{\delta \rightarrow 0}\lim \sup_{x\rightarrow \infty }\frac{x\tilde{I}%
(\delta ,x)}{A(x)} &=&0,\text{ where}  \label{r15} \\
\tilde{I}(\delta ,x) &=&\sum_{1}^{\delta x}p(x-w)\frac{A(w)^{2}}{w}.
\label{r16}
\end{eqnarray}%
We also note that if $B$ denotes any non-negative asymptotically increasing
function then an immediate consequence of (\ref{r15}) is that%
\begin{equation}
\lim_{\delta \rightarrow 0}\lim \sup_{x\rightarrow \infty }\frac{x}{B(\delta
x)A(x)}\sum_{1}^{\delta x}p(x-w)\frac{B(w)A(w)^{2}}{w}=0.  \label{r17}
\end{equation}%
In connection with this we will often use the fact that for any fixed $%
\delta _{0}>0,$%
\begin{eqnarray*}
&&\lim \sup_{x\rightarrow \infty }\frac{x}{B(x)A(x)}\sum_{\delta
_{0}x<w<(1-\delta _{0})x}p(x-w)\frac{B(w)A(w)^{2}}{w} \\
&\leq &\lim \sup_{x\rightarrow \infty }\frac{x}{B(x)A(x)}\frac{B(x)A(x)^{2}}{%
\delta _{0}x}\overline{F}(\delta _{0}x)<\infty ,
\end{eqnarray*}%
and combining this with (\ref{r17}) we see that, whenever (\ref{r15}) holds,%
\begin{equation}
\lim \sup_{x\rightarrow \infty }\frac{x}{B(x)A(x)}\sum_{0<w<(1-\delta
_{0})x}p(x-w)\frac{B(w)A(w)^{2}}{w}<\infty .  \label{r18}
\end{equation}%
We start by showing the necessity of (\ref{r3}). If $f(n,x)$ is any
non-negative function we will write "$f(n,x)$ is asymptotically neglible" to
mean that for some fixed $n_{0}$%
\begin{equation*}
\lim_{\delta \downarrow 0}\lim \sup_{x\rightarrow \infty }\frac{x}{A(x)}%
\sum_{n_{0}}^{\delta A(x)}f(n,x)=0.
\end{equation*}%
We can and will assume henceforth that $n_{0}$ and $\delta $ are chosen so
that when $x$ is large enough the bounds (\ref{l1}) and (\ref{l2}) in
Theorem \ref{LOCAL} are operative. First, we consider possible values of $%
Z_{n}^{(1)}:=\max_{1\leq r\leq n}X_{r},$ and for any fixed $0<\lambda <C$ we
get%
\begin{eqnarray*}
P^{\ast } &:&=P(S_{n}=x,x-Ca_{n}\leq Z_{n}^{(1)}\leq x-\lambda
a_{n}])=\sum_{\lambda a_{n}}^{Ca_{n}}P(S_{n}=x,Z_{n}^{(1)}=x-y) \\
&=&n\sum_{\lambda a_{n}}^{Ca_{n}}P(S_{n-1}=y,Z_{n-1}^{(1)}<x-y,X_{n}=x-y) \\
&=&n\sum_{\lambda a_{n}}^{Ca_{n}}p(x-y)P(S_{n-1}=y,Z_{n-1}^{(1)}<x-y).
\end{eqnarray*}%
By Gnedenko's Local Limit Theorem we see that $a_{n}P(S_{n-1}=y)$ is bounded
above and below by positive constants for $y\in \lbrack \lambda
a_{n},Ca_{n}] $ when $n$ and $\theta =y/a_{n}$ are sufficiently large. We
deduce the bound%
\begin{eqnarray*}
P(S_{n-1} &=&y,Z_{n-1}^{(1)}\geq x-y)\leq n\sum_{x-y}^{\infty
}p(z)P(S_{n-2}=y-z) \\
&\leq &\frac{cn\overline{F}(x-y)}{a_{n}}.
\end{eqnarray*}%
Since $n\overline{F}(x-y)$ can be made small for all $y\leq Ca_{n}$ by
making $\theta $ large, we see that 
\begin{equation*}
P^{\ast }\geq \frac{cn}{a_{n}}\sum_{\lambda a_{n}}^{Ca_{n}}p(x-z).
\end{equation*}%
Thus when $\alpha <1/2$ so that $n/a_{n}\in RV(1-\eta )$ with $1-\eta <-1,$ $%
\sum_{n\in (n_{0},\delta A(x)]}P^{\ast }$ is bounded below by a multiple of 
\begin{eqnarray*}
\sum_{n\in (n_{0},\delta A(x)]}\frac{n}{a_{n}}\sum_{\lambda
a_{n}}^{Ca_{n}}p(x-z) &=&\sum_{Ca(n_{0})}^{C\delta _{x}^{\ast
}x}p(x-z)\sum_{A(z/C)}^{\delta A(x)\wedge A(z/\lambda )}\frac{n}{a_{n}} \\
&\geq &\sum_{Ca(n_{0})}^{\lambda \delta _{x}^{\ast
}x}p(x-z)\sum_{A(z/C)}^{A(z/\lambda )}\frac{n}{a_{n}} \\
&\geq &c\sum_{Ca(n_{0})}^{\lambda \delta _{x}^{\ast }x}p(x-z)\left( \frac{%
A(z/C)^{2}}{z/C}-\frac{A(z/\lambda )^{2}}{z/\lambda }\right) \\
&=&c\tilde{I}(\lambda \delta _{x}^{\ast },x)-c\sum_{1}^{Ca(n_{0})}p(x-z)%
\frac{A(z)^{2}}{z}.
\end{eqnarray*}%
Since it follows from (\ref{r1}) that the final term is asymptotically
neglible, we see that (\ref{r15}) is necessary for the SRT to hold when $%
\alpha <1/2$. In the case $\alpha =1/2$ we write $A(x)=\sqrt{xL(x)}$ and
note that 
\begin{eqnarray*}
\sum_{A(z/C)}^{A(z/\lambda )}\frac{n}{a_{n}} &\backsim
&c\int_{z/C}^{z/\lambda }\frac{A(y)A^{\prime }(y)}{y}dy \\
&\backsim &c\int_{z/C}^{z/\lambda }\frac{A(y)^{2}}{y^{2}}dy=c\int_{z/C}^{z/%
\lambda }\frac{L(y)}{y}dy \\
&=&c\int_{1/C}^{1/\lambda }\frac{L(zw)}{w}dw\backsim cL(z)=c\frac{A(z)^{2}}{z%
}.
\end{eqnarray*}%
Thus (\ref{r15}) is a necessary condition in all cases. To show that it is
also sufficient, we write%
\begin{eqnarray*}
P(S_{n} &=&x)=\sum_{1}^{3}P_{r\text{ }}^{(1)},\text{ where }P_{1\text{ }%
}^{(1)}=P(S_{n}=x,Z_{n}^{(1)}\leq \gamma x), \\
P_{2\text{ }}^{(1)} &=&P(S_{n}=x,Z_{n}^{(1)}\in (\gamma x,x-Ca_{n})),\text{
and } \\
P_{3\text{ }}^{(1)} &=&P(S_{n}=x,Z_{n}^{(1)}\in \lbrack x-Ca_{n},x]).
\end{eqnarray*}%
Note first that there is an upper bound for $P^{\ast }$ corresponding to the
lower bound we established earlier, so the sufficiency for $P_{3}^{(1)}$ to
be asymptotically neglible will follow if we can, given arbitrary $%
\varepsilon >0,$ find a $\lambda >0$ such that 
\begin{equation}
\sum_{n\in (n_{0},\delta A(x)]}\sum_{0}^{\lambda
a_{n}}np(x-z)P(S_{n-1}=z)\leq c\tilde{I}(C\delta _{x}^{\ast },x)+o(A(x)/x).
\label{r20}
\end{equation}%
To see this, we use the following facts, which are contained in Lemma 4 and
the argument leading to (3.16) in \cite{D}: $\exists n_{0},\lambda >0$ such
that%
\begin{eqnarray*}
\text{for }z &\geq &n\geq n_{0}\text{ and }z/a_{n}\leq \lambda \text{ we
have }zP(S_{n-1}=z)\leq ce^{-c(n/A(z))}; \\
\text{for }z &\leq &n\text{ we have }P(S_{n-1}=z)\leq ce^{-cn}.
\end{eqnarray*}%
Splitting the LHS of (\ref{r20}) in the obvious way we we see that it is
bounded above by a multiple of%
\begin{eqnarray*}
&&\sum_{n\in (n_{0},\delta A(x)]}\sum_{0}^{n}ne^{-cn}p(x-z)+\sum_{n\in
(n_{0},\delta A(x)]}\sum_{n}^{\lambda a_{n}}\frac{n}{z}e^{-c(n/A(z))}p(x-z)
\\
&\leq &\sum_{n\in (n_{0},\delta A(x)]}n^{2}e^{-cn}\sup_{1\leq z\leq
n}p(x-z)+\sum_{z=n_{0}}^{\lambda \delta _{x}^{\ast }x}\sum_{n=A(z/\lambda
)}^{z}\frac{n}{z}e^{-c(n/A(z))}p(x-z).
\end{eqnarray*}%
The first term here is $o(A(x)/x)$ by condition (\ref{rz}). Writing $n=A(z)y$
in the second term we see that%
\begin{equation*}
\sum_{n=A(z/\lambda )}^{z}\frac{n}{z}e^{-c(n/A(z))}\backsim \frac{A(z)^{2}}{z%
}\int_{\lambda ^{-\delta \ast }}^{z/A(z)}e^{-cy}dy\leq c\frac{A(z)^{2}}{z},
\end{equation*}%
and then (\ref{r20}) follows.

Next, we choose $\gamma \in (0,\alpha /(1-\alpha ))$ so that by (\ref{l1})
we have $P_{1}^{(1)}\leq \frac{c\{n\overline{F}(x)\}^{1/\gamma }}{a_{n}}$
and since $\frac{n^{1+1/\gamma }}{a_{n}}\in RV(1+1/\gamma -\eta )$ with $%
1+1/\gamma -\eta >0$ it follows that%
\begin{eqnarray*}
\sum_{n\in (n_{0},\delta A(x)]}P_{1}^{(1)} &\leq &c\overline{F}(x)^{1/\gamma
}\sum_{n\in (n_{0},\delta A(x)]}\frac{n^{1/\gamma }}{a_{n}} \\
&\leq &c\overline{F}(x)^{1/\gamma }\frac{(\delta A(x))^{1+1/\gamma }}{\delta
_{x}^{\ast }x}\leq c\frac{\delta ^{1+1/\gamma }A(x)}{\delta _{x}^{\ast }x}.
\end{eqnarray*}%
Recalling that $\delta _{x}^{\ast }\backsim \delta ^{\eta }$ as $%
x\rightarrow \infty ,$ we see that $P_{1}^{(1)}$ is also asymptotically
neglible. This leaves us only to deal with $P_{2}^{(1)}$ and this is more
complicated.

First assume that $\alpha \in (1/3,1/2],$ i.e, $\eta \in \lbrack 2,3),$ so
that $\sum_{1}^{m}\frac{n^{2}}{a_{n}}\backsim c\frac{m^{3}}{a_{m}}$ and we
can assume, WLOG, that $z^{-1}A(z)^{3}$ is increasing$.$ By the same
argument used to get the upper bound for $P_{3}^{(1)},$ but now using the
bound (\ref{l2}), we get that for any $\delta _{0}\in (C\delta _{x}^{\ast
},1-\gamma )$%
\begin{eqnarray}
&&\sum_{n\in (n_{0},\delta A(x)]}P_{2}^{(1)}  \notag \\
&\leq &c\sum_{n\in (n_{0},\delta A(x)]}\frac{n^{2}}{a_{n}}%
\sum_{Ca_{n}}^{(1-\gamma )x}p(x-z)\overline{F}(z)  \notag \\
&=&c\sum_{Ca_{n_{0}}}^{(1-\gamma )x}p(x-z)\overline{F}(z)\sum_{n_{0}}^{%
\delta A(x)\wedge A(z/C)}\frac{n^{2}}{a_{n}}  \notag \\
&\leq &c\sum_{Ca_{n_{0}}}^{(1-\gamma )x}p(x-z)\overline{F}(z)\frac{A(\delta
_{x}^{\ast }x\wedge z/C)^{3}}{\delta _{x}^{\ast }x\wedge z/C}  \notag \\
&\leq &c\sum_{Ca_{n_{0}}}^{\delta _{0}x}p(x-z)\overline{F}(z)\frac{A(z)^{3}}{%
z}+\frac{\delta ^{3}A(x)^{3}}{\delta _{x}^{\ast }x}\sum_{\delta
_{0}x}^{(1-\gamma )x}p(x-z)\overline{F}(z).  \label{r10}
\end{eqnarray}%
Now, given arbitrary $\varepsilon >0,$ we fix $\delta _{0}$ so that $\lim
\sup_{x\rightarrow \infty }\frac{x\tilde{I}(\delta _{0},x)}{A(x)}\leq
\varepsilon .$ Then the second term in (\ref{r10}) is bounded above by%
\begin{equation*}
\frac{\delta ^{3}A(x)^{3}}{\delta _{x}^{\ast }x}\bullet \overline{F}(\gamma
x)\overline{F}(\delta _{0}x)\backsim \frac{\delta ^{3-\eta }A(x)}{(\gamma
\delta _{0})^{\alpha }x}\text{ as }x\rightarrow \infty ,
\end{equation*}%
and it follows that $P_{2}^{(1)}$ is asymptotically neglible. This proves
the theorem for $\alpha \in (1/3,1/2],$ so now we consider other values of $%
\alpha .$

We let $Z_{n}^{(i)},i=1,2,\cdots ,n$ denote the steps $X_{r},r=1,2,\cdots ,n$
arranged in decreasing order, and put $Y_{k}=\sum_{1}^{k}Z_{n}^{(i)}$ for $%
k\geq 1,$ and $Y_{0}=0,$ where we suppress the dependence on $n$. Fix $%
C_{1}>C_{2}>\cdots C_{k}>0.$ Then, if $P_{2}^{(k)}=P(S_{n}=x,B_{k})$ with%
\begin{equation*}
B_{k}=\wedge _{1}^{k}(Z_{n}^{(r)}\in (\gamma
(x-Y_{r-1}),x-Y_{r-1}-C_{r}a_{n}]),
\end{equation*}%
we claim first that $P_{2}^{(k)}$ is asymptotically neglible for all $%
1/(k+2)<\alpha \leq 1/(k+1),$ where $k\geq 2.$ We have%
\begin{eqnarray*}
P_{2}^{(k)} &=&\sum P(Z_{n}^{(1)}=z_{1},\cdots Z_{n}^{(k)}=z_{k},S_{n}=x) \\
&\leq &cn^{k}\sum p(z_{1})\cdots p(z_{k})P(S_{n}=x-(z_{1}+\cdots z_{k})),
\end{eqnarray*}%
where the summation runs over%
\begin{eqnarray*}
z_{1} &\geq &z_{2}\geq \ldots \geq z_{k\text{ }}\geq 0\text{ such that, with 
}z_{0}=0, \\
\text{ }z_{r} &\in &(\gamma (x-(z_{0}+\cdots z_{r-1}),x-(z_{0}+\cdots
z_{r-1})-C_{r}a_{n}])\text{ for }r=1,\cdots k.
\end{eqnarray*}%
Making the change of variable \ $%
x-z_{1}=y_{1},x-(z_{1}+z_{2})=y_{1}-z_{2}=y_{2},\cdots ,$\newline
$x-(z_{1}+\cdots z_{k})=y_{k-1}-z_{k}=y_{k},$ we deduce the bound 
\begin{eqnarray*}
P_{2}^{(k)} &\leq &cn^{k}\sum_{y_{1}=C_{1}a_{n}}^{(1-\gamma )x}\cdots
\sum_{y_{k}=C_{k}a_{n}}^{(1-\gamma )y_{k-1}}p(x-y_{1})\cdots
p(y_{k-1}-y_{k})P(S_{n-k}=y_{k}) \\
&\leq &\frac{cn^{k+1}}{a_{n}}\tsum\nolimits_{k}p(x-y_{1})\cdots
p(y_{k-1}-y_{k})\overline{F}(y_{k}),
\end{eqnarray*}%
where we use (\ref{l2}) and write $\tsum\nolimits_{k}$ as an abbreviation
for the previous sum. (Note we have omitted the requirement $z_{1}\geq
z_{2}\geq \ldots \geq z_{k\text{ }}$ here.) Next, assume $a\neq 1/(k+1),$ so
that $yA(y)^{-r}$ is asymptotically increasing for $r\leq k+1$ and we can
use (\ref{r18}) to get 
\begin{eqnarray*}
\sum_{y_{k}=C_{k}a_{n}}^{(1-\gamma )y_{k-1}}p(y_{k-1}-y_{k})\overline{F}%
(y_{k}) &\leq &c\sum_{y_{k}=C_{k}a_{n}}^{(1-\gamma
)y_{k-1}}p(y_{k-1}-y_{k})A(y_{k})^{2}y_{k}^{-1}\cdot y_{k}A(y_{k})^{-3} \\
&\leq &cy_{k-1}A(y_{k-1})^{-3}A(y_{k-1})/y_{k-1}=cA(y_{k-1})^{-2},
\end{eqnarray*}%
We can then repeat the argument until we get 
\begin{equation}
P_{2}^{(k)}\leq \frac{cn^{k+1}}{a_{n}}\sum_{y_{1}=C_{1}a_{n}}^{(1-\gamma
)x}p(x-y_{1})A(y_{1})^{-k}.  \label{r23}
\end{equation}%
In the case $\alpha =1/(k+1)$ the last step in this procedure requires more
care, since $A(y)^{1-k}=y^{-1}A(y)^{2}\cdot yA(y)^{-k-1}$ and the last
factor here is slowly varying, and not necessarily increasing. But we do have%
\begin{eqnarray}
P_{2}^{(k)} &\leq &\frac{cn^{k+1}}{a_{n}}\sum_{y_{1}=C_{1}a_{n}}^{(1-\gamma
)x}\sum_{y_{2}=C_{2}a_{n}}^{(1-\gamma
)y_{1}}p(x-y_{1})p(y_{1}-y_{2})A(y_{1})^{1-k}  \notag \\
&\leq &\frac{cn^{k+1}}{a_{n}\sqrt{A(C_{2}a_{n})}}%
\sum_{y_{1}=C_{1}a_{n}}^{(1-\gamma )x}\sum_{y_{2}=C_{2}a_{n}}^{(1-\gamma
)y_{1}}p(x-y_{1})p(y_{1}-y_{2})A(y_{1})^{3/2-k}  \notag \\
&\leq &\frac{cn^{k+1/2}}{a_{n}}\sum_{y_{1}=C_{1}a_{n}}^{(1-\gamma
)x}p(x-y_{1})A(y_{1})^{1/2-k},  \label{r24}
\end{eqnarray}%
where we have used the fact that $yA(y)^{-(k+1/2)\text{ }}$is asymptotically
increasing. Furthermore, since $(k+2)\alpha >1,$ we deduce that when (\ref%
{r23}) holds, for any $\delta _{0}\in (C\delta _{x}^{\ast },1-\gamma ),$%
\begin{eqnarray*}
\sum_{n\in (n_{0},\delta A(x)]}P_{2}^{(k)} &\leq &c\sum_{n\in (n_{0},\delta
A(x)]}\sum_{C_{1}a_{n}}^{(1-\gamma )x}\frac{n^{k+1}}{a_{n}}p(x-y)A(y)^{-k} \\
&=&c\sum_{y=C_{1}a_{n_{0}}}^{(1-\gamma
)x}p(x-y)A(y)^{-k}\sum_{n_{0}}^{A(y/C_{1})\wedge \delta A(x)}\frac{n^{k+1}}{%
a_{n}} \\
&\leq &c\sum_{C_{1}a_{n_{0}}}^{(1-\gamma )x}p(x-y)A(y)^{-k}\frac{%
A(C_{1}\delta _{x}^{\ast }x\wedge y)^{k+2}}{C_{1}\delta _{x}^{\ast }x\wedge y%
} \\
&\leq &c\tilde{I}(\delta _{0}x)+\frac{c\delta ^{k+2}A(x)^{k+2}}{\delta
_{x}^{\ast }x}\sum_{\delta _{0}x}^{(1-\gamma )x}p(x-y)A(y)^{-k}.
\end{eqnarray*}%
Since now $yA(y)^{-k-2}$ is asymptotically decreasing we see that the second
term is bounded asymptotically by $c\delta ^{k+2-\eta },$ and so is
asymptotically neglible. If instead we have (\ref{r24}) a slight variation
of this argument gives the same conclusion$.$

Next, we consider $P(S_{n}=x,B_{k}^{c}),$ which we can bound above by $%
\sum_{2}^{k}P_{1}^{(j)}+\sum_{2}^{k}P_{3}^{(j)},$where%
\begin{eqnarray*}
P_{1}^{(j)} &=&P(S_{n}=x,B_{j-1},A_{1}^{(j)}),\text{ with }%
A_{1}^{(j)}=(Y_{j}\leq \gamma (x-Z_{n}^{(j-1)}))\text{, } \\
\text{and }P_{3}^{(j)} &=&P(S_{n}=x,B_{j-1},A_{3}^{(j)}),\text{ with }%
A_{3}^{(j)}=(Z_{n}^{(j)}\in \lbrack x-C_{j}a_{n},x]).\text{ }
\end{eqnarray*}%
Proceeding as above, and using (\ref{l1}), we get the bound%
\begin{eqnarray*}
P_{1}^{(j)}\leq cn^{j-1} &&\tsum\nolimits_{j-1}p(x-y_{1})\cdots
p(y_{j-1}-y_{j-2})P(S_{n-j}=y_{j-1},Z_{n-j}^{(1)}\leq \gamma y_{j-1}) \\
\leq \frac{cn^{j-1+1/\gamma }}{a_{n}} &&\tsum%
\nolimits_{j-1}p(x-y_{1})p(y_{1}-y_{2})\cdots
p(y_{j-2}-y_{j-1})A(y_{j-1})^{-1/\gamma }.
\end{eqnarray*}%
Now 
\begin{eqnarray*}
&&\sum_{y_{j-1}=C_{j-1}a_{n}}^{(1-\gamma
)y_{j-2}}p(y_{j-2}-y_{j-1})A(y_{j-1})^{-1/\gamma } \\
&\leq &c\sum_{y_{j-1}=C_{j-1}a_{n}}^{(1-\gamma
)y_{j-2}}p(y_{j-2}-y_{j-1})A(y_{j-1})^{2}y_{j-1}^{-1}\cdot
y_{j-1}A(y_{j-1})^{-2-1/\gamma } \\
&\leq &cA(y_{k-2})^{-1-1/\gamma },
\end{eqnarray*}%
where we have used (\ref{r18}), which we can do since $1+1/\gamma >\eta .$
Then we can repeat the argument, finally getting%
\begin{eqnarray*}
P_{j}^{(1)} &\leq &\frac{cn^{j-1+1/\gamma }}{a_{n}}%
\sum_{y_{1}=C_{1}a_{n}}^{(1-\gamma )x}p(x-y_{1})A(y_{1})^{-j+2-1/\gamma } \\
&\leq &\frac{cn^{j-1+1/\gamma }}{a_{n}A(x)^{j-1+1/\gamma }},
\end{eqnarray*}%
which gives%
\begin{eqnarray*}
\sum_{n\in (n_{0},\delta A(x)]}P_{j}^{(1)} &\leq &\frac{c}{%
A(x)^{j-1+1/\gamma }}\sum_{n\in (n_{0},\delta A(x)]}\frac{n^{j-1+1/\gamma }}{%
a_{n}} \\
&\leq &\frac{c(\delta A(x))^{j+1/\gamma }}{A(x)^{j-1+1/\gamma }\delta
_{x}^{\ast }x}\leq c\frac{\delta ^{j+1/\gamma }A(x)}{\delta _{x}^{\ast }x}.
\end{eqnarray*}%
Thus the term $P_{j}^{(1)}$ is asymptotically neglible for $2\leq j\leq k$
and $\alpha \leq 1/(k+1).$

Next,%
\begin{eqnarray*}
\sum_{n\in (n_{0},\delta A(x)]}P_{j}^{(3)} &\leq
&c\tsum\nolimits_{j-1}p(x-y_{1})\cdots p(y_{j-2}-y_{j-1})\Theta (y_{j-1}), \\
\text{with }\Theta (y_{j-1}) &=&\sum_{n_{0}}^{\delta A(x)\wedge
A(y_{j-1}/C_{j-1})}\frac{n^{j}}{a_{n}}\sum_{0\leq z\leq
C_{j}a_{n}}p(y_{j-1}-z) \\
&\leq &\sum_{0}^{\gamma _{j}y_{j-1}\wedge C_{j}\delta _{x}^{\ast
}x}p(y_{j-1}-z)\sum_{A(z/C_{j})}^{\delta A(x)\wedge A(y_{j-1}/C_{j})}\frac{%
n^{j}}{a_{n}},
\end{eqnarray*}%
where $\gamma _{j}=C_{j}/C_{j-1}<1.$ When $\alpha <1/(k+1)$ we have $j+1\leq
k+1<\eta ,$ so we can use the bound $\sum_{A(z/C_{j})}^{\infty }\frac{n^{j}}{%
a_{n}}\leq cA(z)^{j+1}/z$ to get%
\begin{eqnarray*}
\Theta (y_{j-1}) &\leq &c\sum_{0}^{\gamma _{j}y_{j-1}\wedge C_{j}\delta
_{x}^{\ast }x}p(y_{j-1}-z)A(z)^{j+1}/z \\
&\leq &cA(y_{j-1}\wedge \delta _{x}^{\ast }x)\sum_{z\leq \gamma
_{j}y_{j-1}}p(y_{j-1}-z)A(z)^{j}/z \\
&\leq &cA(y_{j-1}\wedge \delta _{x}^{\ast }x)\cdot \frac{A(y_{j-1})^{j-1}}{%
y_{j-1}}.
\end{eqnarray*}%
Repeating the process, we are finally left with the $y_{1}$ term, which is%
\begin{eqnarray*}
&&\sum_{y_{1}=C_{0}}^{(1-\gamma )x}p(x-y_{1})A(y_{1}\wedge \delta _{x}^{\ast
}x)\cdot \frac{A(y_{1})}{y_{1}} \\
&\leq &\sum_{y_{1}=C_{0}}^{\delta _{0}x}p(x-y_{1})\frac{A(y_{1})^{2}}{y_{1}}%
+A(C\delta _{x}^{\ast }x)\sum_{y_{1}=\delta _{0}x}^{(1-\gamma )x}p(x-y_{1})%
\frac{A(y_{1})}{y_{1}} \\
&\leq &\tilde{I}(\delta _{0}x)+\frac{A(C\delta _{x}^{\ast }x)}{A(\delta
_{0}x)}\sum_{y_{1}=\delta _{0}x}^{(1-\gamma )x}p(x-y_{1})\frac{A(y_{1})^{2}}{%
y_{1}} \\
&\leq &\tilde{I}(\delta _{0}x)+c\delta \sum_{y_{1}=\delta _{0}x}^{(1-\gamma
)x}p(x-y_{1})\frac{A(y_{1})^{2}}{y_{1}}.
\end{eqnarray*}%
Using (\ref{r18}), we see that $P_{3}^{(j)}$ is asymptotically neglible for $%
2\leq j\leq k.$ If $\alpha =1/(k+1)$ and $j<k$ the same argument works, so
we are left with the case $\alpha =1/(k+1)$ and $j=k,$ which is similar to
the case $\alpha =1/2.$ So again we split $P(S_{n-k}=y_{k-1},Z_{n-k}^{(1)}%
\geq y_{k-1}-C_{k}a_{n})$ into two terms, and estimate $%
P(S_{n-k}=y_{k-1},Z_{n-k}^{(1)}\geq y_{k-1}-\lambda a_{n})$ as before. We
need to deal with the terms%
\begin{eqnarray}
n^{k}\tsum\nolimits_{k-1}p(x-y_{1})\cdots p(y_{k-2}-y_{k-1})\sum_{n\leq
z\leq \lambda a_{n}}z^{-1}e^{-cn/A(z)}p(y_{k-1}-z), &&  \label{r27} \\
\text{and }n^{k}e^{-cn}\tsum\nolimits_{k-1}p(x-y_{1})\cdots
p(y_{k-2}-y_{k-1})\sum_{1\leq z\leq n}p(y_{k-1}-z). &&  \label{r28}
\end{eqnarray}%
We have%
\begin{eqnarray*}
\sum_{n_{0}}^{\delta A(x)}\text{(\ref{r27})} &\leq
&c\sum_{y_{1}=0}^{(1-\gamma )x}\cdots \sum_{y_{k-1}=0}^{(1-\gamma
)y_{k-2}}p(x-y_{1})\cdots p(y_{k-2}-y_{k-1})\Omega (y_{k-1}), \\
\text{where }\Omega (y_{k-1}) &=&\sum_{n_{0}}^{\delta A(x)\wedge
A(y_{k-1}/C_{k-1})}\frac{n^{k}}{a_{n}}\sum_{n\leq z\leq \lambda
a_{n}}z^{-1}e^{-cn/A(z)}p(y_{k-1}-z) \\
&\leq &c\sum_{n_{0}}^{\gamma _{k}y_{k-1}\wedge C_{k}\delta _{x}^{\ast
}x}p(y_{k-1}-z)\sum_{A(z/\lambda )}^{\infty }n^{k}e^{-cn/A(z)} \\
&\leq &c\sum_{n_{0}}^{\gamma _{k}y_{k-1}\wedge C_{k}\delta _{x}^{\ast
}x}p(y_{k-1}-z)\frac{A(z)^{k+1}}{z}.
\end{eqnarray*}%
By repeatedly using the calculation%
\begin{eqnarray*}
\sum_{0}^{\gamma _{k}y_{k-1}\wedge C_{k}\delta _{x}^{\ast }x}p(y_{k-1}-z)%
\frac{A(z)^{k+1}}{z} &\leq &A(\gamma _{k}y_{k-1}\wedge C_{k}\delta
_{x}^{\ast }x)\sum_{0}^{\gamma _{k}y_{k-1}}p(y_{k-1}-z)\frac{A(z)^{k}}{z} \\
&\leq &cA(y_{k-1}\wedge \delta _{x}^{\ast }x)\frac{A(y_{k-1})^{k-1}}{y_{k-1}}%
,
\end{eqnarray*}%
we deduce that%
\begin{equation*}
\sum_{n_{0}}^{\delta A(x)}\text{(\ref{r27})}\leq c\sum_{y_{1}=0}^{(1-\gamma
)x}p(x-y_{1})\frac{A(y_{1})A(y_{1}\wedge \delta _{x}^{\ast }x)}{y_{1}}.
\end{equation*}%
This deals with (\ref{r27}), and to bound (\ref{r28} ) we use the fact that $%
n=o(a_{n})$ to see that (\ref{r1}) implies that for sufficiently large $n$
we have 
\begin{eqnarray*}
\sum_{1\leq z<n}p(y_{k-1}-z) &\leq &\frac{cnA(y_{k-1})}{y_{k-1}} \\
&=&\frac{cA(y_{k-1})^{2}}{y_{k-1}}\cdot \frac{n}{A(y_{k-1})} \\
&\leq &\frac{cA(y_{k-1})^{2}}{y_{k-1}}\text{ for }y_{k-1}\geq C_{k-1}a_{n},
\end{eqnarray*}%
so that 
\begin{eqnarray*}
&&\sum_{y_{k-1}=C_{k-1}a_{n}}^{(1-\gamma
)y_{k-2}}p(y_{k-2}-y_{k-1})\sum_{1\leq z<n}p(y_{k-1}-z) \\
&\leq &c\sum_{y_{k-1}=C_{k-1}a_{n}}^{(1-\gamma )y_{k-2}}p(y_{k-2}-y_{k-1})%
\frac{A(y_{k-1})^{2}}{y_{k-1}} \\
&\leq &\frac{cA(y_{k-2})}{y_{k-2}}\leq \frac{cA(y_{k-2})^{2}}{ny_{k-2}}\text{
for }y_{k-2}\geq C_{k-2}a_{n}.
\end{eqnarray*}%
Repeating this we deduce that, for any $\varepsilon >0,$%
\begin{eqnarray*}
\sum_{n_{0}}^{\delta A(x)}\text{(\ref{r28})} &\leq &c\sum_{n_{0}}^{\delta
A(x)}ne^{-cn}\sum_{y_{1}=C_{1}a_{n}}^{(1-\gamma )x}p(x-y_{1})\frac{%
A(y_{1})^{2}}{y_{1}} \\
&\leq &\frac{cA(x)}{x}\sum_{n_{0}}^{\infty }ne^{-cn}\leq \frac{\varepsilon
A(x)}{x},
\end{eqnarray*}%
provided $n_{0}$ is chosen sufficiently large. This concludes the proof.
\end{proof}

\section{Extensions of the renewal process results}

\begin{enumerate}
\item As previously remarked, in the non-lattice case the obvious analogue
of Theorem \ref{srt} holds. By this we mean that the NASC for (\ref{r2}) to
hold with $g(x)$ replaced by $G(x,\Delta ):=\sum_{1}^{\infty }P(S_{n}\in
(x,x+\Delta ])$ and $g(\alpha ,\rho )$ replaced by $\Delta g(\alpha ,\rho )$
for any fixed $\Delta >0$ is that both%
\begin{equation*}
\lim_{x\rightarrow \infty }x\overline{F}(x)P(S_{1}\in (x,x+\Delta ])=0,
\end{equation*}%
and 
\begin{equation*}
\lim_{\delta \rightarrow 0}\lim \sup_{x\rightarrow \infty }x\overline{F}%
(x)\int_{1}^{\delta x}\frac{F(x-dw)}{w\overline{F}(w)^{2}}=0.
\end{equation*}

\item Similarly, whenever $F$ has a density and a density version of the
local limit therem holds for $S_{n}$ a density version of Theorem \ref{srt}
can be proved in the same manner.

\item In \cite{D}, Theorem 3 contains an extension of the SRT to generalized
Green's functions of the form%
\begin{equation*}
g_{b}(x)=\sum_{0}^{\infty }b_{n}P(S_{n}=x),
\end{equation*}%
where $b$ is a non-negative function which is regularly varying at $\infty $
of index $\beta .$ That result was obtained under the restriction 
\begin{equation}
\sup_{x\geq 1}\omega (x)<\infty ,\text{ where }\omega (x):=xp(x)/\overline{F}%
(x).  \label{g4}
\end{equation}
Here we show that (\ref{g4}) is redundant, by giving a NASC for the same
result.
\end{enumerate}

\begin{theorem}
\label{Gen}i) Assume that $F$ is aperiodic, $P(X\geq 0)=1,$ $S\in D(\alpha
,1)$ with $\alpha \in (0,1)$ and $\beta >-2.$ Put $b(A(\cdot ))=B(\cdot ).$
Then when $\alpha (2+\beta )>1$ 
\begin{equation}
\lim_{x\rightarrow \infty }\frac{x\overline{F}(x)g_{b}(x)}{B(x)}=g(\alpha
,,\beta ):=\alpha E(Y^{-\alpha (\beta +1)}),  \label{g1}
\end{equation}%
where $Y$ denotes a random variable having the limiting stable law. When $%
\alpha (2+\beta )\leq 1$ (\ref{g1}) holds if and only if, for every fixed $%
n_{0}$%
\begin{equation}
\lim_{x\rightarrow \infty }\frac{x\overline{F}(x)p(x)}{B(x)}=0,  \label{g2}
\end{equation}%
and 
\begin{equation}
\lim_{\delta \rightarrow 0}\lim \sup_{x\rightarrow \infty }\frac{x\overline{F%
}(x)}{B(x)}\sum_{1}^{\delta x}\frac{p(x-w)B(w)}{w\overline{F}(w)^{2}}=0.
\label{g3}
\end{equation}
\end{theorem}

\begin{proof}
A careful reading of the proof of Theorem 3 in \cite{D} shows that the only
way (\ref{g4}) is used is in establishing%
\begin{equation*}
\lim_{\delta \rightarrow 0}\lim \sup_{x\rightarrow \infty }\frac{x}{A(x)B(x)}%
\sum_{n\in (n_{0},\delta A(x)]}b_{n}P(S_{n}=x)=0,
\end{equation*}%
which is the analogue of (\ref{r5}). The proof of this, without (\ref{g4}),
essentially amounts to repeating the proof of (\ref{r5}) with the difference
that wherever we dealt with sums involving $n^{j}/a_{n}$ we now need to deal
with $n^{j}b(n)/a_{n};$ so the difficulties associated with integer values
of $\eta $ become associated with integer values of $\eta -\beta .$ The
details are omitted.
\end{proof}

\begin{remark}
It should also be mentioned that the restriction on the value of $\beta $ is
necessary for $g(\alpha ,\beta )$ to be finite.
\end{remark}

\begin{remark}
It should be noted that when (\ref{g4}) holds, both (\ref{g2}) and (\ref{g3}%
) are automatic.
\end{remark}

\section{Random walks}

An obvious question is whether the result for renewal processes in part (ii)
of Theorem \ref{srt} extends to the random walk case. We make the following

\begin{conjecture}
The conditions (\ref{rz}) and (\ref{r3}) of Theorem \ref{srt} are necessary
and sufficient for the SRT (\ref{r2}) to hold for any aperiodic random walk
in $D(\alpha ,\rho )$ with $\alpha \in (0,1)$ and $\rho >0.$
\end{conjecture}

\begin{remark}
In principle, a variation of our method should establish this result. In
fact the proof of the necessity of the conditions is straightforward.
Likewise the proof of the sufficiency when $\alpha >1/3$ is not difficult.
Specifically we write $P(S_{n}=x)=\sum_{1}^{4}P_{r\text{ }}^{(1)},$ where $%
P_{1\text{ }}^{(1)}$ and $P_{2\text{ }}^{(1)}$ are as before, $P_{3\text{ }%
}^{(1)}=P(S_{n}=x,Z_{n}^{(1)}\in \lbrack x-Ca_{n},x+Ca_{n}]\},$ and $P_{4%
\text{ }}^{(1)}=P(S_{n}=x,Z_{n}^{(1)}>x+Ca_{n}).$ If we note that (\ref{r3})
is actually equivalent to 
\begin{equation*}
\lim_{\delta \rightarrow 0}\lim \sup_{x\rightarrow \infty }x\overline{F}%
(x)\sum_{-\delta x}^{\delta x}\frac{p(x-w)}{w\overline{F}(w)^{2}}=0,
\end{equation*}%
the estimate of $P_{3\text{ }}^{(1)}$ requires only minor changes. Finally a
similar argument, but using the result (\ref{l2}) for $-S,$ gives%
\begin{eqnarray*}
\sum_{n\in (n_{0},\delta A(x)]}P_{4}^{(1)} &\leq &c\sum_{Ca_{n_{0}}}^{\delta
_{0}x}p(x+z)F(-z)\frac{A(z)^{3}}{z} \\
&&+\frac{\delta ^{3}A(x)^{3}}{\delta ^{\ast }}\sum_{\delta _{0}x}^{\infty
}p(x+z)F(-z).
\end{eqnarray*}%
It then follows from a slight variation of the argument following (\ref{r10}%
) that%
\begin{equation*}
\lim_{\delta \rightarrow 0}\lim \sup_{x\rightarrow \infty }\frac{x}{A(x)}%
\sum_{n\in (n_{0},\delta A(x)]}P_{4}^{(1)}=0,
\end{equation*}%
which completes the proof. However a proof when $\alpha \in
((k+2)^{-1},(k+1)^{-1})$ for general $k$ seems to require consideration of
the $k$ steps which are largest in modulus, and this seems quite complicated.
\end{remark}

\section{Subordinators}

Our proof of Theorem \ref{srt} rests on the classical local limit theorems,
Theorem \ref{LOCAL}, and consideration of a finite number of the largest
jumps. It is not difficult to see that each of these items can be replicated
for an asymptotically stable subordinator, and then essentially the same
argument leads to the following result, whose proof is omitted.

\begin{theorem}
Let $X$ be any subordinator that is in the domain of attraction of a stable
law of index $\alpha \in (0,1)$ as $t\rightarrow \infty ,$ and define it's
renewal measure by 
\begin{equation*}
G(dx)=\int_{0}^{\infty }P(X_{t}\in dx)dt
\end{equation*}%
Suppose also that its L\'{e}vy measure is non-lattice, and write 
\begin{equation*}
G_{\Delta }(x)=G((x,x+\Delta ]).
\end{equation*}%
Then for any fixed $\Delta >0,$ (i) if $\alpha >1/2,$ 
\begin{equation}
\lim_{x\rightarrow \infty }x\overline{\Pi }(x)G_{\Delta }(x)=\Delta g(\alpha
,\rho )=\alpha \Delta E(Y^{-\alpha }),  \label{x1}
\end{equation}%
where $Y$ denotes a random variable having the limiting stable law. (ii) if $%
\alpha \in (0,1/2]$ then (\ref{x1}) holds if and only if%
\begin{equation}
\lim_{x\rightarrow \infty }x\overline{\Pi }(x)\Pi ((x,x+\Delta ])=0
\label{x2}
\end{equation}%
and 
\begin{equation}
\lim_{\delta \rightarrow 0}\lim \sup_{x\rightarrow \infty }x\overline{\Pi }%
(x)\int_{1}^{\delta x}\frac{\Pi (x-dw)}{w\overline{\Pi }(w)^{2}}=0.
\label{x3}
\end{equation}
\end{theorem}

\begin{acknowledgement}
Almost simultaneously a different proof of the main result of this paper has
appeared in Caravenna, \cite{Car}.
\end{acknowledgement}

\end{document}